\documentclass[final,a4paper,reqno,11pt]{amsart}
\pdfoutput=1

\usepackage{soul}
\usepackage[final]{graphicx}
\usepackage[final]{hyperref}
\usepackage{amssymb, amstext, amsmath, amsthm, amsfonts, mathrsfs, lmodern, ifdraft, stmaryrd, enumitem}
\usepackage{tikz}
\usepackage{tikz-cd}
\usetikzlibrary{decorations.pathmorphing}
\usepackage[noadjust]{cite}
\usepackage[utf8]{inputenc}
\usepackage[english]{babel}
\usepackage[noabbrev]{cleveref}
\usepackage{wasysym}
\usepackage{float}
\usepackage[activate={true,nocompatibility},final,tracking=true,kerning=true,spacing=true,factor=1100,stretch=10,shrink=10]{microtype}
\microtypecontext{spacing=nonfrench}
\usepackage[figurename=Diagram]{caption}
\usepackage{setspace, caption}
\usepackage{xspace, mathtools, todonotes}
\newtheorem{theorem}{Theorem}[section]

\newtheorem{corollary}[theorem]{Corollary}
\newtheorem{lemma}[theorem]{Lemma}
\newtheorem{proposition}[theorem]{Proposition}

\newtheorem*{theorem*}{Theorem}
\newtheorem*{corollary*}{Corollary}
\newtheorem{introtheorem}{Theorem}

\Crefname{introtheorem}{Theorem}{Theorems}

\theoremstyle{definition}
\newtheorem{definition}[theorem]{Definition}

\newtheorem{setup}[theorem]{Setup}

\newtheorem*{conjecture*}{Conjecture}

\theoremstyle{remark}
\newtheorem{remark}[theorem]{Remark}
\newenvironment{clm}{\emph{Claim.}}{}
\newenvironment{clmproof}{\emph{Proof of the Claim.}}{}

\newtheorem*{acknowledgment}{Acknowledgment}

\makeatletter
\newenvironment{sketch}[1][Sketch]{\par
\pushQED{\qed}%
\normalfont \topsep6\p@\@plus6\p@\relax
\trivlist
\item\relax
{\itshape
#1\@addpunct{.}}\hspace\labelsep\ignorespaces
}{%
\popQED\endtrivlist\@endpefalse
}
\makeatother

\newcommand{\cC}{\mathcal{C}}

\newcommand{\cE}{\mathcal{E}}

\newcommand{\cP}{\mathcal{P}}

\newcommand{\sX}{\mathscr{X}}

\newcommand{\sF}{\mathscr{F}}

\newcommand{\sG}{\mathscr{G}}

\newcommand{\kk}{\Bbbk}
\newcommand{\NN}{\mathbb{N}}

\newcommand{\ZZ}{\mathbb{Z}}

\newcommand{\Ab}{\mathsf{Ab}}

\renewcommand{\mod}{\mathsf{mod}\hspace{.01in}}
\newcommand{\ex}{\mathsf{ex}\hspace{.01in}}
\newcommand{\exn}{\mathsf{ex}_n\hspace{.01in}}
\newcommand{\exnC}{\mathsf{ex}_n\hspace{.01in}\cC}
\newcommand{\exnCop}{\mathsf{ex}_n\hspace{.01in}\cC^{\op}}
\newcommand{\PCn}{\operatorname{Fun}^{\cC}_n}
\newcommand{\PCnop}{\operatorname{Fun}^{\cC^{\op}}_n}
\newcommand{\RCn}{\operatorname{Res}^{\cC}_n}
\newcommand{\OCn}{\operatorname{Op}^{\cC}_n}
\newcommand{\OCnop}{\operatorname{Op}^{{\cC}^{\op}}_n}
\newcommand{\XnC}{\mathbf{C}^{\mathrm{ex}}_n\hspace{.01in}(\cC)}
\newcommand{\XnCop}{\mathbf{C}^{\mathrm{ex}}_n\hspace{.01in}(\cC^{\op})}
\newcommand{\KXnC}{\mathbf{K}^{\mathrm{ex}}_n\hspace{.01in}(\cC)}
\newcommand{\KC}{\mathbf{K}\hspace{.01in}(\cC)}
\newcommand{\K}{\mathbf{K}\hspace{.01in}}
\newcommand{\KXnCop}{\mathbf{K}^{\mathrm{ex}}_n\hspace{.01in}(\cC^{\op})}
\newcommand{\nC}{\mathbf{C}_n\hspace{.01in}(\cC)}

\newcommand{\Mod}{\mathsf{Mod}\hspace{.01in}}
\newcommand{\proj}{\mathsf{proj}\hspace{.01in}}

\newcommand{\cok}{\operatorname{Cok}\nolimits}
\newcommand{\cokm}[1]{p_{\hspace{.001in}#1}}

\newcommand{\Ext}{\operatorname{Ext}\nolimits}

\newcommand{\Hm}{\operatorname{H}\nolimits}
\newcommand{\Z}{\operatorname{Z}\nolimits}

\newcommand{\Hom}{\operatorname{Hom}\nolimits}
\newcommand{\id}{\operatorname{id}\nolimits}
\newcommand{\Id}{\operatorname{Id}\nolimits}
\newcommand{\im}{\operatorname{Im}\nolimits}
\newcommand{\imm}[1]{j_{#1}}
\newcommand{\coimm}[1]{q_{#1}}
\renewcommand{\ker}{\operatorname{Ker}\nolimits}
\newcommand{\kerm}[1]{i_{\hspace{.001in}#1}}
\newcommand{\op}{\operatorname{op}\nolimits}
\newcommand{\pdim}{\operatorname{pdim}\nolimits}

\newcommand{\nker}{left $n$-exact sequence\xspace}

\newcommand{\ncoker}{right $n$-exact sequence\xspace}
\newcommand{\ncokers}{right $n$-exact sequences\xspace}
\newcommand{\aug}{\cokm{\cov{d^X_1}}}

\renewcommand{\subset}{\subseteq}
\makeatletter
\def\Biggg#1{{\hbox{$\left#1\vbox to25\p@{}\right.\n@space$}}}
\makeatother

\newcommand{\Tr}{\operatorname{Tr}}
\newcommand{\trnc}{{\Tr_{n+1}^{\cC}}}
\newcommand{\trncop}{{\Tr_{n+1}^{\cC^{\op}}}}

\newcommand{\Pb}{\mathsf{Pb}}

\newcommand{\Po}{\mathsf{Po}}
\newcommand{\sm}[1]{\left[\begin{smallmatrix} #1 \end{smallmatrix}\right]}

\newcommand{\s}{\mbox{}\xspace}
\newcommand{\coloneqq}{:=}
\crefname{figure}{diagram}{diagrams}
\Crefname{figure}{Diagram}{Diagrams}

\makeatletter
 \newcommand{\xmapsfrom}[2][]{%
    \ext@arrow3095\leftarrowfill@{#1}{#2}\mapsfromchar
}
\makeatother

\makeatletter
\def\Axioms#1{\expandafter\@Axioms\csname c@#1\endcsname}
\def\@Axioms#1{\ifcase#1\or (EA1)\or (EA1${}^{\op}$)\or (EA2)\or (EA2${}^{\op}$)\fi}

\AddEnumerateCounter{\Axioms}{\@Axioms}{First}
\makeatother

\numberwithin{equation}{section}

\newcommand{\covnp}[2][]{{#2}_{\cC}^{#1}}
\newcommand{\connp}[2][]{{#2}_{#1}^{\cC}}
\newcommand{\cov}[2][]{\covnp[{#1}]{[#2]}}
\newcommand{\con}[2][]{\connp[{#1}]{[#2]}}
\newcommand{\covsm}[2][]{\covnp[#1]{\sm{#2}}}

\makeatletter
\newcommand{\mylabel}[2]{#2\def\@currentlabel{#2}\label{#1}}
\makeatother

\begin{document}
   \title[The existence and uniqueness of maximal $n$-exact structures]{On $n$-exact categories I: The existence and uniqueness of maximal $n$-exact structures}

   \author{Carlo Klapproth}\address{Institut für Algebra und Zahlentheorie (IAZ), Universität Stuttgart, 70569 Stuttgart, Germany} 
   \email{carlo.klapproth@mathematik.uni-stuttgart.de}
   \maketitle
   
   \begin{abstract}
         This paper is the first part of a series that investigates the existence of $n$-exact structures on idempotent complete additive categories for positive integers $n$.
         It is shown that every idempotent complete additive category has a unique maximal $n$-exact structure.
         We achieve this by constructing a bijection between $n$-exact structures on a category and certain subcategories of its functor category following ideas of Enomoto.
   \end{abstract}

\tableofcontents

   \section{Introduction}
   Jasso introduced $n$-exact categories in \cite{Jas16} to formalise the properties which are inherited by \emph{$n$-cluster-tilting subcategories} of abelian or exact categories.
   Their axiomatic framework is very similar to that of exact categories but $3$-term short exact sequences are replaced by \emph{$n$-exact sequences} which have $(n+2)$-terms, cf.\s\ e.g.\s\ \Cref{def:nker}.
   In particular, the axioms of $n$-exact categories are intrinsic, that is they do not depend on any ambient abelian or exact category.
   Though it is not clear from the onset whether an additive category carries a non-trivial $n$-exact structure and indeed most examples of $n$-exact categories seem to be constructed as $n$-cluster-tilting subcategories of exact categories.
   We follow Enomoto's functorial approach to exact categories \cite{Eno18} to give a more intrinsic approach of classifying all $n$-exact structures on a given idempotent complete category $\cC$.
   This approach is conceptually also similar to Gulisz's recent functorial approach to $n$-abelian categories, cf.\s\ \cite{Gul24}.

   We show that $n$-exact sequences in $\cC$ can be seen as projective resolutions of functors which are contained in a subcategory $\exnC$ of the category of all finitely resolved functors $\mod \cC$, cf.\s\ \Cref{not:functors}.
   Moreover, we develop a modified set of axioms for $n$-exact categories which is equivalent to Jasso's original definition but which is based on a class of conflations which is closed under homotopy equivalence instead of weak isomorphisms, see \Cref{def:nexact} and \Cref{thm:thesame}.
   The above ideas allow us to view $n$-exact structures on $\cC$ as extension closed subcategories of $\exnC$ which are stable under two operations $\Pb^\cC_n$ and $\Po^{\cC}_n$, cf.\s\ \Cref{def:exnc,def:exncop}.
   \begin{introtheorem}[{\Cref{thm:main}}]\label{introthm:main}
      For any idempotent complete additive category $\cC$ there is an inclusion preserving one-to-one correspondence 
      \begin{align*}
         \left\{\parbox{13.2em}{\centering $n$-exact structures $(\cC,\sX)$ on $\cC$}\right\} &\xleftrightarrow{\text{\emph{1:1}}} {\left\{\parbox{18em}{\centering extension closed subcategories $\sF \subset \exn \cC$ such that $\Pb_n^{\cC}(\sF) = \sF$ and $\Po_n^{\cC}(\sF) = \sF$}\right\}}.
      \end{align*}
   \end{introtheorem}
   Stability under $\Pb^\cC_n$ resp.\s\ $\Po^\cC_n$ corresponds to deflations resp.\s\ inflations being stable under weak pullbacks resp.\s\ weak pushouts.
   Under these stability assumptions, being extension closed corresponds to inflations and deflations being closed under composition.

   We show, as a corollary of Enomoto's approach, that every idempotent complete additive category carries a unique maximal $n$-exact structure.
   Similar results in the case $n=1$ have been achieved by Sieg \cite{Sie11} for pre-abelian categories and more generally by Crivei \cite{Cri12} for weakly idempotent complete additive categories and by Rump \cite{Rum11} for arbitrary additive categories.
   \begin{introtheorem}[\Cref{thm:exmax}]\label{introthm:main2}
      Every idempotent complete additive category $\cC$ admits a unique maximal $n$-exact structure. 
   \end{introtheorem}
   In our language it turns out that the operations $\Pb^\cC_n$ and $\Po^\cC_n$ turn extension closed subcategories of $\exnC$ into such.
   As a consequence of this and \Cref{introthm:main} we achieve \Cref{introthm:main2} by repeatedly applying the two operators to $\exnC$ and then intersecting. 
    
   The unique maximal $n$-exact structure on $\cC$ can be trivial, that is it may consist only of the split $n$-exact sequences.
   In the second part to this paper we will use \Cref{introthm:main,introthm:main2} to construct a series of explicit non-degenerate examples arising from $(n+1)$-dimensional objects of geometric origin.

\section{Conventions and notation}
    Throughout this paper we assume the following global setup unless stated otherwise. 
\begin{setup}\label{stp:global}
   Let $n$ be a positive integer and $\cC$ be an idempotent complete additive category.
\end{setup}
 
    Given two morphisms $X \xrightarrow{f} Y \xrightarrow{g} Z$ in a category we denote their composite by $gf$.
    If $F$ is an additive functor between additive categories and $\sX$ is contained in the domain of $F$ then we denote by $F(\sX)$ the essential image of $\sX$ under $F$.
    If $X_\bullet$ is a complex in the domain of $F$ we denote by $F(X_\bullet)$ the complex obtained from applying $F$ termwise to $X_\bullet$.

\subsection{\texorpdfstring{$\cC$}{C}-Modules and \texorpdfstring{$\cC$}{C}-morphisms}
In this section we recall well-known facts about functor categories.
We want to mitigate the set-theoretic issues associated with the category $\Mod \cC$ of all additive functors $\cC^{\op} \to \Ab$ for a potentially non-small category $\cC$.
However, if $\Mod \cC$ is a well-defined category, for example if $\cC$ is small, then $\Mod \cC$ is an abelian category and most of the statements below are direct consequences of this.

We say a \emph{$\cC$-module} is an additive functor $F \colon \cC^{\op} \to \Ab$ and a \emph{$\cC$-morphism} is a natural transformation $\alpha \colon F \to G$ of $\cC$-modules.
We write $\cov{=} \coloneqq \cC(-,=)$ and $\con{-} \coloneqq \cC(-,=)$ for the covariant resp.\s\ contravariant Yoneda embedding.
Moreover, for $X \in \cC$ we write $\cov[X]{=} \coloneqq \cC(X,=)$ and $\con[X]{-} \coloneqq \cC(-,X)$. 

We call a sequence $F \to G \to H$ of $\cC$-modules \emph{pointwise exact at $G$} if the induced sequence $F(X) \to G(X) \to H(X)$ is exact at $G(X)$ in $\Ab$ for all $X \in \cC$.
A $\cC$-morphism $\alpha \colon F \to G$ is called \emph{pointwise monomorphism} (or \emph{pointwise epimorphism}) if $\alpha_X$ is a monomorphism (or epimorphism) in $\Ab$ for all $X \in \cC$.

Through a calculation one can show that the following definition is correct.
\begin{definition}
    For any $\cC$-morphism $\alpha \colon F \to G$ we define a $\cC$-module $\cok \alpha$, called the \emph{cokernel of $\alpha$}, and a $\cC$-morphism $\cokm{\alpha} \colon G \to \cok \alpha$ as follows:
\begin{enumerate}[label={{(\alph*)}}]
    \item For $X \in \cC$ the abelian group $(\cok \alpha)(X)$ is the quotient $G(X)/\im(\alpha_X)$.
    \item For $X \in \cC$ the morphism $(\cokm{\alpha})_X \colon G(X) \to (\cok \alpha)(X)$ is the canonical map.
    \item For $f \in \cC(X,Y)$ the map $(\cok \alpha)(f) \colon (\cok \alpha)(Y) \to (\cok \alpha)(X)$ is the unique morphism in $\Ab$ which satisfies $(\cok \alpha)(f) (\cokm{\alpha})_Y = (\cokm{\alpha})_X G(f)$.
\end{enumerate} 
Similarly, we define $\cC$-modules $\ker \alpha$ and $\im \alpha$, called the \emph{kernel of $\alpha$} resp.\s\ \emph{image of $\alpha$}, and $\cC$-morphisms $\kerm{\alpha} \colon \ker \alpha \to F$ as well as $\coimm{\alpha} \colon F \to \im \alpha$ and $\imm{\alpha} \colon \im \alpha \to G$ with $\alpha = \imm{\alpha}\coimm{\alpha}$, by using the canonical kernels resp.\s\ images in $\Ab$.
We also construct direct sums, pushouts and pullbacks of $\cC$-modules as we do in $\Ab$.
We sometimes write $G/\im \alpha \coloneqq \cok \alpha$.%
\end{definition}

The above $\cC$-modules and $\cC$-morphisms are illustrated in the commutative \Cref{fig:imagefact}
\begin{figure}[H]
\begin{tikzcd}[ampersand replacement=\&,cramped]
	\& F \&\& G \\
	{\ker \alpha} \&\& {\im \alpha} \&\& {\cok \alpha}
	\arrow["\alpha", from=1-2, to=1-4]
	\arrow["{q_{\alpha}}", from=1-2, to=2-3]
	\arrow["{p_{\alpha}}", from=1-4, to=2-5]
	\arrow["{i_{\alpha}}", from=2-1, to=1-2]
	\arrow["0", from=2-1, to=2-3]
	\arrow["{j_{\alpha}}", from=2-3, to=1-4]
	\arrow["0", from=2-3, to=2-5]
\end{tikzcd}
\caption{Canonical $\cC$-modules and $\cC$-morphisms associated to $\alpha$.}
\label{fig:imagefact}
\end{figure}\noindent
where $0 \to \ker \alpha \to F \to \im \alpha \to 0$ and $0 \to \im \alpha \to G \to \cok \alpha \to 0$ form pointwise exact sequences.
We say that $\imm\alpha\coimm \alpha$ is an \emph{image factorisation of $\alpha$}.
The following useful lemma is referred to as the \emph{universal property of the kernel}.

\begin{lemma}\label{lem:kern}
    Let $0 \to F \xrightarrow{\alpha} G \xrightarrow{\beta} H$ be a pointwise exact sequence of $\cC$-modules and $\gamma$ be a $\cC$-morphism ending at $G$. 
    If $\beta \gamma = 0$ then $\gamma$ factors uniquely through $\alpha$.
\end{lemma}
\begin{proof}
    This follows from the universal property of the kernel in $\Ab$ using that $\alpha_X$ is a monomorphism for all $X \in \cC$ to show that the obtained family of factorisations is natural.
\end{proof}
The dual property to \Cref{lem:kern} is called the \emph{universal property of the cokernel} and it also holds.
The following direct consequence of the Yoneda lemma implies that $\cov{X}$ is an projective object in $\Mod \cC$ if the latter is a well-defined category.
\begin{lemma}\label{lem:projlift}
    Let $F \xrightarrow{\alpha} G \xrightarrow{\beta} H$ be a pointwise exact sequence of $\cC$-modules, $X$ an object in $\cC$ and $\gamma \colon \cov{X} \to G$ be a $\cC$-morphism with $\beta \gamma = 0$.
    Then $\gamma$ factors through $\alpha$.
\end{lemma}

The following can be seen as an extension of the third isomorphism theorem.

\begin{lemma}\label{lem:rsnake}
For any pair of $\cC$-morphisms $\alpha \colon F \to G$ and $\beta \colon G \to H$ there exists an induced pointwise exact sequence
$0 \to {\im \cokm{\beta\alpha} \beta}  \to {\cok \beta \alpha} \to {\cok \beta} \to 0$
of $\cC$-modules. Moreover, $\cok \alpha/\im \cokm\alpha\kerm\beta \cong \im \cokm{\beta\alpha} \beta$ naturally.
\end{lemma}
\begin{proof}
    We are given the solid morphisms of a commutative diagram
\[\begin{tikzcd}[column sep = {5em,between origins}, row sep = {3.5em,between origins}]
	{\ker\beta\alpha} && {\ker\beta} & \\
	{\ker\alpha} & F & G & {\cok \alpha} \\
	0 & & H & \\
	& 0 & {\cok \beta} & {\cok \beta \alpha}.
	\arrow[dotted, from=1-1, to=1-3]
	\arrow["{\kerm{\beta\alpha}}", from=1-1, to=2-2]
	\arrow["{\kerm \beta}"', from=1-3, to=2-3]
	\arrow["{\cokm \alpha \kerm \beta}", dotted, from=1-3, to=2-4]
	\arrow[dotted, from=2-1, to=1-1]
	\arrow["{\kerm{\alpha}}"', from=2-1, to=2-2]
	\arrow["\alpha", from=2-2, to=2-3]
	\arrow["{\beta\alpha}"', from=2-2, to=3-3]
	\arrow["{\cokm \alpha}"', from=2-3, to=2-4]
	\arrow["\beta", from=2-3, to=3-3]
	\arrow[dotted, from=2-4, to=4-4]
	\arrow[dotted, from=3-1, to=2-1]
	\arrow["{\cokm \beta}"', from=3-3, to=4-3]
	\arrow["{\cokm{\beta \alpha}}", from=3-3, to=4-4]
	\arrow[dotted, from=4-3, to=4-2]
	\arrow[dotted, from=4-4, to=4-3]
\end{tikzcd}\]
    Using composition and the universal properties of the kernel and the cokernel we can construct the dotted arrows such that the diagram is commutative. 
    A calclulation in $\Ab$ shows that the dotted sequence is pointwise exact.
    Thus the kernel of the arrow $\cok \beta\alpha \to \cok \beta$ is equal to the image of the arrow $\cok \alpha \to \cok \beta \alpha$ which agrees with $\im \cokm{\alpha\beta}\beta $ as $\cokm{\alpha}$ is a pointwise epimorphism.
    Pointwise exactness at $\cok \alpha$ shows that the kernel of $\cok \alpha \to \cok \beta \alpha$ is equal to $\im \cokm \alpha \kerm \beta$.
    The universal property of the cokernel thus shows that the image of the arrow $\cok \alpha \to \cok \beta \alpha$ is isomorphic to $\cok \alpha/\im {\cokm{\alpha}\kerm{\beta}}$.
\end{proof}

    The following lemma is useful when dealing with weak pullbacks in $\cC$.

\begin{lemma}\label{lem:dia}
    Given the solid $\cC$-morphisms in a diagram
    \[\begin{tikzcd}
	F & G \\
	{G'} & H \\
    \im \cokm{\beta}\beta' & \cok \beta
	\arrow["\alpha", dotted, from=1-1, to=1-2]
	\arrow["{\alpha'}"', dotted, from=1-1, to=2-1]
	\arrow["\beta", from=1-2, to=2-2]
	\arrow["{\beta'}", from=2-1, to=2-2]
    \arrow["\coimm{\cokm\beta\beta'}"', from=2-1, to=3-1, dotted]
    \arrow["\cokm{\beta}", from=2-2, to=3-2]
    \arrow["\imm{\cokm\beta\beta'}", from=3-1, to=3-2, dotted]
\end{tikzcd}\]
   there exists a natural isomorphism $\cok \sm{-\beta & \beta'} \cong \cok \beta/\im (\cokm\beta\beta')$. Additionally, if the dotted $\cC$-morphisms are given such that 
   \[\begin{tikzcd}[ampersand replacement=\&, column sep = 4em]
	F \& {G\oplus G'} \& H
	\arrow["\begin{array}{c} \sm{\alpha \\ \alpha'} \end{array}", from=1-1, to=1-2]
	\arrow["{\sm{-\beta & \beta'}}", from=1-2, to=1-3]
\end{tikzcd}\] 
    is pointwise exact at $G \oplus G'$, then there is a natural isomorphism $\im(\cokm\beta \beta') \cong \cok \alpha'$.
\end{lemma}
\begin{proof}
    The first part follows from \Cref{lem:rsnake} applied to $\sm{-\id_G & 0}^\top$ and $\sm{-\beta & \beta'}$ noticing that $\im \cokm{\beta} \sm{-\beta & \beta'} = \im \cokm{\beta} \beta'$ as $\cokm{\beta}\beta = 0$.
    The second part follows from showing that $\ker \coimm{\cokm\beta\beta'} = \im \alpha'$ for example through a pointwise diagram chase in $\Ab$.
\end{proof}

Similar to \cite[Defininition 2.2]{Eno17}, we introduce the following category which consists of all \emph{finitely resolved $\cC$-modules}.
Notice, all set-theoretic issues can be eliminated because $\mod \cC$ can be identified with a subcategory of the homotopy category $\KC$ of $\cC$.

\begin{definition}\label{lem:fresfunctors}
We denote by $\mod \cC$ the category of all $\cC$-modules $F$ for which there exists an (infinite) pointwise exact sequence of the form $\cdots \to \cov{X_1} \to \cov{X_0} \to F \to 0$ where $X_i \in \cC$ for $i \in \NN$ equipped with $\cC$-morphisms as morphisms and the usual composition.
\end{definition}

The category $\mod \cC$ satisfies the following \Cref{lem:syzygy,cor:exactcat}. 
In the case where $\cC$ is essentially small this is a direct consequence of \cite[Proposition 2.6]{Eno17} and the proof given in loc.\s\ cit.\s\ can be adopted to the case where $\cC$ is potentially not essentially small.

\begin{lemma}\label{lem:syzygy}
   Suppose $0 \to F \to G \to H \to 0$ is a pointwise exact sequence of $\cC$-modules. 
   If two of the three $\cC$-modules $F$, $G$ and $H$ are in $\mod \cC$ then so is the third.
\end{lemma}

In the following lemma the assumption that $\cC$ is idempotent complete is crucial to show that $\proj \cC$ consists up to isomorphism precisely of the representable functors.

\begin{lemma}\label{cor:exactcat}
    Let $\sX$ be the class of all pointwise short exact sequences in $\mod \cC$.
    Then $(\mod \cC, \sX)$ is an idempotent complete exact category with enough projective objects $\proj \cC$, which are up to natural equivalence precisely the representable functors $\cov{X}$ for $X \in \cC$.
\end{lemma}

    We will make use of the fact that giving a $\cC$-morphism $\cov{X} \to \cov{Y}$ is the same as giving a morphism $X \to Y$ for any pair $X,Y \in \cC$ because $\cC \cong \proj \cC$ via $\cov{-} \colon \cC \xrightarrow{\sim} \proj \cC$.
    We have the following corollary of the description of the projective dimension in exact categories via the vanishing of the $\Ext$-functor and the long exact $\Ext$-sequence.
\begin{lemma}\label{lem:syzygy2}
    If $0 \to F \to \cov{X} \to G \to 0$ is a pointwise exact sequence in $\mod \cC$ with $X \in \cC$ and $G \notin \proj \cC$ then $\pdim_{\mod \cC} F = \pdim_{\mod \cC} G - 1$, where $\infty - 1 \coloneqq \infty$. 
\end{lemma}

Finally, we occasionally use the following well-known definition.

\begin{definition}[Cycles, boundaries and homology]
    Given a complex $P_\bullet$ in $\K(\proj \cC)$ we can factorise $\smash{{d^P_i} = \kerm{{d^P_{i-1}}} \beta^P_i}$ for any $i \in \ZZ$, where $\beta^P_i \colon X_i \to \ker d^P_{i-1}$ is uniquely determined by the universal property of the kernel.
    For $i \in \ZZ$ we define 
    \begin{align*}
        \operatorname{Z}_i {P_\bullet} &\coloneqq \ker d^P_{i}, & \operatorname{B}_i {P_\bullet} &\coloneqq \im d^P_{i+1}, & \Hm_i {P_\bullet} &\coloneqq \cok \beta^P_i.
    \end{align*}
    Similarly we define $Z^i P^\bullet$, $B^i P^\bullet$ and $H^i P^\bullet$ for a cocomplex $P^\bullet$ in $\K(\proj \cC)$.
\end{definition}

\subsection{\texorpdfstring{$n$}{n}-Exact categories}
In this section we recall important notions concerning $n$-exact categories.
We also introduce the categories $\XnC$, $\KXnC$ and $\exnC$ as well as the functors $\PCn$, $\RCn$ and $\OCn$ which are crucial in \Cref{sec:funapproach}.

Recall the notions of \emph{$n$-kernels} and \emph{$n$-cokernels} from \cite[Definition 2.2]{Jas16}.
To simplify the notation we introduce the following definitions, cf.\s\ also \cite[Definition 2.4]{Jas16}.
\begin{definition}\label{def:nker}
A \emph{\nker} in $\cC$ is a complex
\begin{equation}\label{eq:nexact}
    \begin{tikzcd}[column sep = {5em,between origins}]
        X_\bullet \mathrlap{\colon} & X_{n+1} \ar[r, "d^X_{n+1}"] & X_n \ar[r, "d^X_n"] & \cdots \ar[r, "d^X_{2}"] & X_1 \ar[r, "d^X_1"] & X_{0}
    \end{tikzcd}
\end{equation}
such that $(d^X_{n+1}, \dots, d^X_{2})$ is an $n$-kernel of $d^X_1$.
Dually, we define \emph{\ncokers}.
A complex which is a left and \ncoker is called an \emph{$n$-exact sequence}.
\end{definition}

\begin{definition}
   Let $\XnC$ be the category of all $n$-exact sequences in $\cC$ with morphisms given by morphisms of complexes and the usual composition. 
   Let $\KXnC$ be the homotopy category of $\XnC$, that is the ideal quotient of $\XnC$ by the null-homotopic morphisms.
\end{definition}

    Recall that $\XnC$ and $\KXnC$ have the same objects.
    Furthermore, $\XnC$ is closed under homotopy equivalence in the category $\nC$ of all $n$-term complexes, i.e.\s\ of all complexes of the shape as \eqref{eq:nexact}, by \cite[Proposition 2.5]{Jas16}.
    Thus $\KXnC$ is closed under isomorphism in the homotopy category $\mathbf{K}_n(\cC)$ of $\nC$.
    It will be convenient for us to identify classes of $n$-exact sequences which are closed under homotopy equivalence in $\nC$ with isomorphism closed classes of $\KXnC$. 

    Recall that a morphism $f \in \cC(X, Y)$ may be interpreted as a morphism $f \in \cC^{\op}(Y,X)$ and that we have $\con{f} = \cC(f,-) = \cC^{\op}(-,f) = [f]_{\cC^{\op}}$ in $\mod \cC^{\op}$.
    As the definition of $n$-exact sequences is self-dual, we have the following.
\begin{lemma}\label{lem:op}
    By formally reversing the direction of all arrows and shifting, we obtain a duality $\XnC \to \XnCop$ which induces a duality $\OCn \colon \KXnC \to \KXnCop$. 
    Moreover, $\OCn$ and $\OCnop$ are mutually inverse to each other.
\end{lemma}
\begin{sketch}
    Explicitly, the duality maps a complex $X_\bullet \in \XnC$ with differential $d^X_\bullet$ to the complex $X_{n + 1 - \bullet} \in \XnCop$ with differential $d^X_{n +2 -\bullet}$ and a morphism $f_\bullet \colon X_\bullet \to Y_\bullet$ in $\XnC$ to $f_{n+1-\bullet} \colon Y_{n+1-\bullet} \to X_{n+1-\bullet}$ in $\XnCop$.
    By the self dual nature of $n$-exact sequences this is well-defined and it is clear that an inverse $\XnCop \to \XnC$ is given by the same formulas.
    A direct calculation shows that these functors are additive, preserve homotopy equivalences and thus descend to dualities $\KXnC \leftrightarrow \KXnCop$.
\end{sketch}

\begin{remark}\label{rem:cohomology}
    By writing all definitions out explicitly we have 
       \begin{align*} \Hm^{n+1} \con[Y]{X_\bullet} = \cok \con[Y]{d^X_{n+1}} = \cok [d^X_{n+1}]_{\cC^{\op}}^Y = \Hm_0 [\OCn(X_\bullet)]^Y_{\cC^{\op}} \end{align*}
    naturally in $X_\bullet \in \KXnC^{\op}$ and $Y \in \cC$, using that $X_\bullet$ vanishes in all degrees above $n+1$.
\end{remark}    

    We want to interpret $n$-exact sequences as projective resolutions of certain functors and therefore need the following lemma, which is similar to \cite[Proposition 3.4]{Gul24}.

   \begin{lemma}\label{lem:nexact}
      Let $X_\bullet$ be a complex in $\nC$ and $F \coloneqq \cok \cov{d^X_1}$. Then the statements 
      \begin{enumerate}[label={{(\alph*)}}]
         \item $F \in \mod \cC$ and $\cov{X_\bullet}$ is a projective resolution of $F$ in $\mod \cC$, and\label{item:projres}
         \item $X_\bullet$ is a \nker\label{item:nkernel}
      \end{enumerate}
      are equivalent. 
      If the statement \ref{item:projres} holds then the statements
      \begin{enumerate}[resume, label={{(\alph*)}}]
         \item $\Ext_{\mod \cC}^i(F, \cov{Y}) = 0$ for all $i = 0, \dots, n$ and $Y \in \cC$, and \label{item:allext}
         \item $X_\bullet$ is an $n$-exact sequence\label{item:nexact}
      \end{enumerate}
      are equivalent as well.
   \end{lemma}
   \begin{proof}
        Applying $\cov{-}$ to $X_\bullet$ we obtain a complex
      \begin{equation}\label{eq:YonedaEmbedding}
         \begin{tikzcd}[column sep = 3em]
            0 \ar[r] &[-1em] \cov{X_{n+1}} \ar[r, "{\cov{d^X_{n+1}}}"] & \cov{X_n} \ar[r, "{\cov{d^X_n}}"] &[-.3em] \cdots \ar[r, "{\cov{d^X_{2}}}"] &[.4em] \cov{X_1} \ar[r, "{\cov{d^X_1}}"] & \cov{X_{0}} \ar[r] &[-1em] 0
         \end{tikzcd}
      \end{equation}
      of projective functors in $\mod \cC$. 
      This is a projective resolution of $F$, which is then in $\mod \cC$, if and only if it is pointwise exact everywhere except possibly at $\cov{X_{0}}$.
      This is the case if and only if $X_\bullet$ is a \nker, by definition. 

      Applying $\Hom_{\mod \cC}(-, \cov{Y})$ and the Yoneda lemma to \eqref{eq:YonedaEmbedding}, we obtain a cocomplex
      \begin{equation}\label{eq:YonedaCoEmbedding}
         \begin{tikzcd}[column sep = 3em]
            0 &[-1em] \ar[l] \con[Y]{X_{n+1}} & \ar[l, "{\con[Y]{d^X_{n+1}}}"'] \con[Y]{X_n} &[-.3em] \ar[l, "{\con[Y]{d^X_n}}"'] \cdots &[.4em] \ar[l, "{\con[Y]{d^X_{2}}}"'] \con[Y]{X_1} & \ar[l, "{\con[Y]{d^X_1}}"'] \con[Y]{X_{0}} &[-1em] \ar[l] 0\rlap{.} 
         \end{tikzcd}
      \end{equation}
      Under the assumption that $\cov{X_\bullet}$ is a projective resolution of $F$ the sequence \eqref{eq:YonedaCoEmbedding} is exact at $\con[Y]{X_{i}}$ if and only if $\Ext_{\mod \cC}^{i}(F, \cov{Y}) = 0$.
      This is true for all $Y \in \cC$ and $i = 0,\dots,n$ if and only if $X_\bullet$ is a \ncoker, by definition. 
      Since $X_\bullet$ is already a \nker, this is the case if and only if $X_\bullet$ is an $n$-exact sequence. 
   \end{proof}

   We now want to identify the class $\KXnC$ of all $n$-exact sequences in $\cC$ with the class of functors satisfying all conditions in \Cref{lem:nexact}.

   \begin{definition}\label{not:functors} 
        We denote by $\exn \cC \subset \mod \cC$ the subcategory of all $F \in \mod \cC$ with $\pdim_{\mod \cC} F \leq n+1$ and $\Ext_{\mod \cC}^{i}(F, \cov{Y}) = 0$ for all $0 \leq i \leq n$ and all $Y \in \cC$.
   \end{definition}

    By the horseshoe lemma and the long exact $\Ext$-sequences $\exnC \subset \mod \cC$ is closed under extensions.
    In particular, the class of all pointwise exact sequences forms an exact structure on $\exnC$.
    Moreover, the following remark shows that all non-zero objects $F \in \exn \cC$ have projective dimension exactly $n+1$.

   \begin{remark}\label{rem:pdim}
      In any exact category $\cE$ with enough projective objects $\cP$ any object $M \in \cE$ with $\pdim_{\cE} M \leq n+1$ and $\Ext_{\cE}^i(M,P)$ for all $P \in \cP$ and all $0 \leq i \leq n$ satisfies either $M = 0$ or $\pdim_{\cE} M = n+1$.
      Indeed, if $0 \to P_m \to \dots \to P_0 \to M \to 0$ is an augmented projective resolution of minimal length, where $m \geq 0$ and $P_m \neq 0$, the equivalence class of $\id_{P_m} \in \Z^m [P_\bullet]^{\cE}_{P_m}$ defines a non-zero element $\Ext_{\cE}^m(M,P_m) = \Hm^m [P_\bullet]^{\cE}_{P_m}$. 
   \end{remark}
   
The dual of \Cref{lem:kern} and \Cref{lem:nexact} immediately yield the following result.

  \begin{lemma}\label{lem:resolve}
    There is a functor $\PCn\colon \KXnC \to \exn \cC$ sending an $n$-exact sequence $X_\bullet$ to $\cok \cov{d^X_1}$ and a morphism $f_\bullet \colon X_\bullet \to Y_\bullet$ to the unique $\alpha \colon \cok \cov{d^X_1} \to \cok \cov{d^Y_1}$ which satisfies $\smash{\alpha \cokm{\cov{d^X_{1}}} = \cokm{\cov{d^Y_1}} \cov{f_{0}}}$.
  \end{lemma}

    \begin{remark}\label{rem:homology}
        We have $\PCn(X_\bullet)(Y) = \Hm_0(\cov[Y]{X_\bullet})$ naturally in $X_\bullet \in \KXnC$ and $Y\in \cC^{\op}$ using that $X_\bullet$ vanishes in negative degrees.
    \end{remark}
    Recall that projective resolutions and lifts of morphisms to projective resolutions exist and are unique up to homotopy equivalence.
    Thus we have the following.
   \begin{lemma}\label{lem:resex}
    There is a functor $\RCn \colon \exn \cC \to \KXnC$ sending  a functor $F \in \exnC$ to an $n$-exact sequence $X^F_\bullet$ such that there is an augmented projective resolution 
      \begin{equation*}
         \begin{tikzcd}[column sep = 3em]
            0 \ar[r] &[-1.5em] \cov{X^F_{n+1}} \ar[r, "{\cov{d^{X^F}_{n+1}}}"] & \cov{X^F_n} \ar[r, "{\cov{d^{X^F}_n}}"] &[-.3em] \cdots \ar[r, "{\cov{d^{X^F}_{2}}}"] &[.4em] \cov{X^F_1} \ar[r, "{\cov{d^{X^F}_1}}"] & \cov{X^F_{0}} \ar[r, "\varepsilon_F"] &[-1em] F \ar[r] &[-1.5em] 0
         \end{tikzcd}
      \end{equation*}
    and a morphism $\alpha \colon F \to G$ to a morphism $f^\alpha_\bullet \colon X^F_\bullet \to X^G_\bullet$ such that $\cov{f^\alpha_\bullet}$ is a lift of $\alpha$, that is $\alpha \varepsilon_F = \varepsilon_G \cov{f^\alpha_0}$.
    Any such functor $\RCn$ is quasi-inverse to $\PCn$. 
    \end{lemma}
    \begin{sketch}
        The homotopy categories $\KC$ and $\K(\proj \cC)$ are equivalent via $\cov{-}$ by the Yoneda lemma.
        The constructions in \cite[Proposition 12.2 and Theorem 12.4]{Bueh10} yield a functor $\exnC \to \KC$ as the exact category $\mod \cC$ has enough projective objects given by $\proj \cC$.
        We may choose the projective resolutions such that the image of the constructed functor is in $\KXnC\subset \KC$ using  $\pdim_{\mod \cC} F \leq n+1$ for all $F \in \exnC$ and \Cref{lem:nexact}.

        A calculation using the universal property of the cokernel shows $\PCn \circ \RCn \cong \Id_{\exnC}$ and \cite[Theorem 12.4 and Corollary 12.5]{Bueh10} show that $\RCn \circ \PCn \cong \Id_{\KXnC}$.
    \end{sketch}

The following notions were originally introduced as \emph{$\sX$-admissible $n$-exact sequences}, \emph{$\sX$-admissible monomorphisms} and \emph{$\sX$-admissible epimorphisms} in \cite[Section 4.1]{Jas16}. 

\begin{definition}\label{def:confinfdef}
   Let $\sX \subset \KXnC$ be a class of $n$-exact sequences.
   \begin{enumerate}[label={{(\alph*)}}]
      \item An \emph{$\sX$-conflation} is a complex 
      which is in $\sX$.
      \item An \emph{$\sX$-inflation} (or \emph{$\sX$-deflation}) is a map $d^X_{n+1} \colon X_{n+1} \to X_n$ (or $d^X_1 \colon X_1 \to X_{0}$) that appears as the first (or last) morphism of an $\sX$-conflation $X_\bullet$ as in \eqref{eq:nexact}. 
   \end{enumerate}
\end{definition}

    Notice, if $\sX$ is closed under isomorphism in $\KXnC$ then it follows from \Cref{lem:resex} that a morphism $d^X_1$ is an $\sX$-deflation if and only if $\cok \cov{d^X_1} \in \PCn(\sX)$.
    Moreover, we have the following.
\begin{remark}\label{rem:isoclos}
    The class of $\sX$-inflations and $\sX$-deflations is closed under isomorphism for any class $\sX$ of $n$-exact sequences, which is closed under isomorphism in $\KXnC$ or even just closed under isomorphism in $\XnC$.
\end{remark}

Recall, a quasi abelian category is a category which has all kernels and cokernels and on which the class of all kernel-cokernel pairs forms an exact structure.
Inspired by this we introduce the following notion which we will show to naturally appear in examples arising from commutative rings.
   \begin{definition}
      We say that $\cC$ is \emph{quasi $n$-abelian} if all morphisms of $\cC$ admit both an $n$-kernel and an $n$-cokernel and moreover the class of all $n$-exact sequences of $\cC$ is an $n$-exact structure on $\cC$.
   \end{definition}

   \begin{remark}\label{rem:abimplqab}
        If $\cC$ is $n$-abelian then $\cC$ is quasi $n$-abelian, by \cite[Definition 3.1(A1)]{Jas16} and \cite[Theorem 4.4]{Jas16}. 
        On the other hand, $n$-abelian categories can be characterised as those quasi $n$-abelian categories $\cC$ which have the property that every finitely presented $\cC$-module (resp.\s\ $\cC^{\op}$-module) with no non-zero $\cC$-morphisms (resp.\s\ $\cC^{\op}$-morphisms) to representable functors lies in $\exnC$ (resp.\s\ $\exnCop$).
        In fact, $\exnC$ containing all finitely presented additive functor which have no non-zero $\cC$-morphisms to any representable functor is equivalent to \cite[Definition 3.1(A2${}^{\op}$)]{Jas16} by using an argument similar to \Cref{lem:nexact} to characterize the functors in question as those which are isomorphic to $\cok \cov{f}$ for an epimorphism $f$ in $\cC$.
        The dual observation then shows the above characterisation of $n$-abelian categories.
    \end{remark}
    
    In \cite[Theorem 3.6]{Gul24} a related characterization of $n$-abelian categories has been given.
    Their axiom (F1${}^{\op}$) characterises Jasso's axiom (A1${}^{\op}$) functorially while our description includes no such characterisation.
    The main difference between their axiom (F2${}^{\op}$) and our description of Jasso's axiom (A2${}^{\op}$) is that in loc.\s\ cit.\s\ epimorphisms are implicitly characterised as morphisms $f$ with $\pdim_{\mod \cC^{\op}} \cok \con{f} \leq 1$. 
    A translation between their condition of $n$-torsion-freeness and our $\Ext$-vanishing condition is then possible by using \cite[Proposition 3.4]{Gul24}.
    Similar considerations apply to their axioms (F1) and (F2).

\section{Main results}
\subsection{Weak isomorphisms vs.\s\ homotopy equivalences of \texorpdfstring{$n$}{n}-exact sequences}
Jasso introduced the notion of $n$-exact categories in \cite{Jas16}. 
In \Cref{def:nexact} below we give a different set of axioms with the major difference that we ask the class of $n$-exact sequences to be closed under homotopy equivalence instead of weak isomorphisms.
We show in \Cref{thm:thesame} that the two notions coincide.
This equivalent definition has also been independently observed by Kvamme \cite{Kva24}. 

\begin{definition}\label{def:nexact}
   Let $\sX \subset \KXnC$ be a class which is closed under isomorphism.
   Then $(\cC,\sX)$ is an $n$-exact category if the following list of axioms is satisfied.
   \begin{enumerate}
      \item[\mylabel{EA0}{(Ex0)}] The complex $0 \to 0 \to \cdots \to 0 \to 0$ is in $\sX$.
      \item[\mylabel{EA1}{(Ex1)}] The class of $\sX$-inflations is closed under composition.
      \item[\mylabel{EA1op}{(Ex1${}^{\op}$)}] The class of $\sX$-deflations is closed under composition.
      \item[\mylabel{EA2}{(Ex2)}] If there are two morphisms $d^X_{n+1}$ and $f_{n+1}$ with the same domain
      \begin{equation}\label{eq:PO}\begin{tikzcd}[ampersand replacement = \&]
      	{X_{n+1}} \& {X_{n}} \\
      	{Y_{n+1}} \& {Y_{n}},
      	\arrow["{d^X_{n+1}}", tail, from=1-1, to=1-2]
      	\arrow["{f_{n+1}}"', from=1-1, to=2-1]
      	\arrow["{d^Y_{n+1}}", dotted, tail, from=2-1, to=2-2]
      	\arrow["{f_{n}}"', dotted, from=1-2, to=2-2]
      \end{tikzcd}\end{equation}
      where $d^X_{n+1}$ is also an $\sX$-inflation, then $\sm{-d^X_{n+1} & f_{n+1}}^\top$ is an $\sX$-inflation admitting a weak cokernel $\sm{f_n & d^Y_{n+1}}$ such that $d^Y_{n+1}$ is an $\sX$-inflation.
      \item[\mylabel{EA2op}{(Ex2${}^{\op}$)}] If there are two morphisms $d^X_1$ and $f_{0}$ with the same codomain
      \begin{equation}\label{eq:PB}\begin{tikzcd}[ampersand replacement = \&]
      	{Y_{1}} \& {Y_{0}} \\
      	{X_{1}} \& {X_{0}}, 
      	\arrow["{d^Y_1}", two heads, dotted, from=1-1, to=1-2]
      	\arrow["{f_1}"', dotted, from=1-1, to=2-1]
      	\arrow["{d^X_1}", two heads, from=2-1, to=2-2]
      	\arrow["{f_{0}}"', from=1-2, to=2-2]
      \end{tikzcd}\end{equation}
      where $d^X_1$ is also an $\sX$-deflation, then $\sm{-d^X_1 & f_{0}}$ is an $\sX$-deflation admitting a weak kernel $\sm{f_1 & d^Y_1}^\top$ such that $d^Y_1$ is an $\sX$-deflation.
   \end{enumerate}
\end{definition}

\begin{remark}
    The dotted arrows make the squares \eqref{eq:PO} and \eqref{eq:PB} commute, because the composite of a morphism with any of its weak cokernels resp.\s\ weak kernels vanishes.
    Moreover, the conditions in \ref{EA2} and \ref{EA2op} imply that the diagrams \eqref{eq:PO} and \eqref{eq:PB} are weak pushout resp.\s\ weak pullback diagrams.
\end{remark}
   
Using $\OCn$ from \Cref{lem:op} we have the following remark. 

\begin{remark}\label{rem:opclear}
    An isomorphism closed class $\sX \subset \KXnC$ satisfies \ref{EA2} resp.\s\ \ref{EA1} in $\cC$ if and only if $\OCn(\sX) \subset \KXnCop$ satisfies \ref{EA2op} resp.\s\ \ref{EA1op} in $\cC^{\op}$.
    Indeed, as $\OCn$ formally reverses arrows this follows from weak kernels being dual to weak cokernels.
\end{remark} 

To show that for (idempotent complete) additive categories this definition agrees with the original one, we need the following lemma.

\begin{lemma}\label{lem:mktom-1k}
    Suppose $d^X_1 \colon X_{1} \to X_{0}$ admits an $n$-kernel and for some $1 \leq m \leq n$ there is a sequence
    \begin{equation}\begin{tikzcd}
	{X_m} & {X_{m-1}} & \cdots & {X_1} & {X_{0}},
	\arrow["{d^X_m}", from=1-1, to=1-2]
	\arrow["{d^X_{m-1}}", from=1-2, to=1-3]
	\arrow["{d^X_{2}}", from=1-3, to=1-4]
	\arrow["{d^X_{1}}", from=1-4, to=1-5]
    \end{tikzcd}\label{fig:start}\end{equation}
    where $d^X_i$ is a weak kernel of $d^X_{i-1}$ for $2 \leq i \leq m$.
    Then $d^X_m$ admits an $(n+1-m)$-kernel, that is \eqref{fig:start} can be completed to a \nker. 
\end{lemma}
\begin{proof}
    Because $d^X_1$ has an $n$-kernel, $\cok\cov{d^X_1}$ is in $\mod \cC$ and has projective dimension at most $n+1$, by \Cref{lem:nexact}.
    By applying $\cov{-}$ to \eqref{fig:start} we obtain a sequence
    \[\begin{tikzcd}[column sep = 3.2em]
	{\cov{X_m}} &[-.5em] {\cov{X_{m-1}}} & \cdots &[-.5em] {\cov{X_1}} &[-.5em] {\cov{X_{0}}} & {\cok{\cov{d^X_1}}} \ar[r] &[-2em] 0
	\arrow["{\cov{d^X_m}}", from=1-1, to=1-2]
	\arrow["{\cov{d^X_{m-1}}}", from=1-2, to=1-3]
	\arrow["{\cov{d^X_{2}}}", from=1-3, to=1-4]
	\arrow["{\cov{d^X_{1}}}", from=1-4, to=1-5]
	\arrow["\aug", from=1-5, to=1-6]
    \end{tikzcd}\]
    which is pointwise exact. 
    Recursive application of \Cref{lem:syzygy,lem:syzygy2} shows that $\ker \cov{d^X_m}$ is in $\mod \cC$ and has projective dimension at most $n-m$. 
    Picking a projective resolution of $\ker \cov{d^X_m}$ and using \Cref{lem:nexact} and the Yoneda embedding being an equivalence shows that $d^X_{m}$ has an $(n+1-m)$-kernel.
\end{proof}

   In the following two lemmas notice that our axioms \ref{EA0}, \ref{EA1} and \ref{EA1op} are already the same as (E0), (E1) and (E1${}^{\op}$) in \cite[Definition 4.2]{Jas16}.
\begin{lemma}\label{lem:thesame}
   An $n$-exact category $(\cC,\sX)$ in the sense of \cite[Definition 4.2]{Jas16} is also an $n$-exact category in the sense of \Cref{def:nexact}. 
\end{lemma}
\begin{proof}
   It follows from \mbox{\cite[Lemma 4.22(4) and Proposition 4.23]{HLN21}} the class $\sX$ is closed under homotopy equivalences in $\nC$ and thus closed under isomorphism in $\KXnC$.

   We only show that \cite[Defintion 4.2(E2${}^{\op}$)]{Jas16} implies \ref{EA2op}, as the statement that \cite[Defintion 4.2(E2)]{Jas16} implies \ref{EA2} follows dually.
   Suppose we are given the solid morphisms of \eqref{eq:PB}, with $d^X_1$ being an $\sX$-deflation. 
   Then \cite[Definition 4.2(E2${}^{\op}$)]{Jas16} and the dual of \cite[Proposition 4.8]{Jas16} show that $\sm{-d^X_1 & f_{0}}$ is an $\sX$-deflation having a weak kernel of the form
   \begin{equation*} 
        \sm{-d^X_{2} & f_{1} \\ 0 & d^Y_1} \colon X_{2} \oplus Y_1 \to X_1 \oplus Y_{0},
   \end{equation*}
   where $d^Y_1$ is an $\sX$-deflation and $d^X_{2}$ is a weak kernel of $d^X_1$.
   By the same argument applied to the $\sX$-deflation $d^Y_1 \colon Y_{1} \to Y_{0}$ and the morphism $0 \colon X_{2} \to Y_{0}$ it follows that the morphism $\sm{0 & d^Y_1} \colon X_{2} \oplus Y_1 \to Y_{0}$ is also an $\sX$-deflation, cf.\s\ also \Cref{rem:isoclos}. Thus
   \begin{equation*}\begin{tikzcd}[ampersand replacement = \&, column sep = large]
      	{ X_{2} \oplus Y_{1}} \& {Y_{0}} \\
      	{X_{1}} \& {X_{0}} 
      	\arrow["{\sm{0 & d^Y_1}}", two heads, from=1-1, to=1-2]
      	\arrow["{\sm{-d^X_2 & f_1}}"', from=1-1, to=2-1]
      	\arrow["{d^X_1}", two heads, from=2-1, to=2-2]
      	\arrow["{f_{0}}"', from=1-2, to=2-2]
      \end{tikzcd}\end{equation*}
   gives a possible choice for the desired dotted arrows in \eqref{eq:PB}.
\end{proof}

\begin{lemma}\label{lem:thesame2}
   An $n$-exact category $(\cC,\sX)$ in the sense of \Cref{def:nexact} is also an $n$-exact category in the sense of \cite[Definition 4.2]{Jas16}.
\end{lemma}
 \begin{proof} 
   We first show that $\sX$ is closed under weak isomorphisms of $n$-exact sequences.
   By \cite[Proposition 2.7]{Jas16} it is enough to show that if 
   \[\begin{tikzcd}
      Y_\bullet \ar[d, "f_\bullet"] & Y_{n+1} \ar[d, "f_{n+1}"] \ar[r, "d^Y_{n+1}"]  & Y_n \ar[r, "d^Y_n"] \ar[d, "f_n"] & \cdots \ar[r, "d^Y_{2}"] & Y_1 \ar[d, "f_1"] \ar[r, "d^Y_1"] & Y_{0} \ar[d, "f_{0}"]\\
      X_\bullet & X_{n+1} \ar[r, "d^X_{n+1}"] & X_n \ar[r, "d^X_n"] & \cdots \ar[r, "d^X_{2}"] & X_1 \ar[r, "d^X_1"] & X_{0}
   \end{tikzcd}\]
   is a morphism of $n$-exact sequences, with $f_0$ and $f_{n+1}$ isomorphisms, then $X_\bullet \in \sX$ if and only if $Y_\bullet \in \sX$.
   We may assume $Y_\bullet \in \sX$, the other case is dual. 
   By \ref{EA2op} applied to $Y_\bullet$ and the morphism $f^{-1}_{0} d^X_1 \colon X_1 \to Y_{0}$ we have that $\sm{-d^Y_1 & f^{-1}_{0} d^X_1}$ is an $\sX$-deflation.
   There is a commutative diagram with two vertical solid isomorphisms
   \[\begin{tikzcd}[ampersand replacement=\&, column sep = large, row sep = 3em]
      {X_{n+1}} \&[-1.3em] \cdots \&[-1.3em] {X_{3}} \&[-.4em] {Y_{1} \oplus X_{2}} \&[1.3em] {Y_1 \oplus X_1} \&[2.5em] {Y_{0}} \\
   	{Z_{n+1}} \& \cdots \& {Z_{3}} \& {Z_{2}} \& {Y_1 \oplus X_1} \& {Y_{0}},
   	\arrow[from=2-1, tail, to=2-2]
   	\arrow[from=2-2, to=2-3]
    \arrow["{\sm{-d^Y_1 & f^{-1}_{0}d^X_1}}", two heads, from=2-5, to=2-6]
   	\arrow[from=2-3, to=2-4]
   	\arrow[from=2-4, to=2-5]
   	\arrow[dotted, from=1-1, to=2-1]
   	\arrow[dotted, from=1-3, to=2-3]
   	\arrow[dotted, from=1-4, to=2-4]
   	\arrow["{d^X_{n+1}}", from=1-1, to=1-2]
   	\arrow["{d^X_{4}}", from=1-2, to=1-3]
   	\arrow["{\sm{\id_{Y_1}& 0 \\ 0 & d^X_{2}}}", from=1-4, to=1-5]
   	\arrow["{\sm{\id_{Y_1} & 0 \\ f_1 & \id_{X_1}}}"', from=1-5, to=2-5]
   	\arrow["{\sm{0 & d^X_1}}", from=1-5, to=1-6]
      \arrow["f^{-1}_{0}"', from=1-6, to=2-6]
   	\arrow["{\sm{0 \\ d^X_{3}}}", from=1-3, to=1-4]
   \end{tikzcd}\]
   where the upper row is the sum of the complexes $0 \to \cdots\to 0 \to Y_1 = Y_1 \to 0$ and $X_\bullet$, and the lower row is an $\sX$-conflation.
   The upper complex is homotopy equivalent to $X_\bullet$ and thus an $n$-exact sequence. %
   Therefore, our diagram can be completed with the dotted arrows into a morphism of $n$-exact sequences, which is a homotopy equivalence by \cite[Proposition 2.7]{Jas16}.
   It follows that $X_\bullet$ is homotopy equivalent to the lower row which is an $\sX$-conflation and hence $X_\bullet \in \sX$.

   We only show that \ref{EA2op} implies \cite[Defintion 4.2(E2${}^{\op}$)]{Jas16}. 
   That \ref{EA2} implies \cite[Defintion 4.2(E2)]{Jas16} follows dually.
   Suppose we are given a conflation $X_\bullet$ and a morphism $f_{0} \colon Y_{0} \to X_{0}$ as illustrated in the diagram
      \begin{equation}\label{eq:npullback}
         \begin{tikzcd}[column sep = {6em,between origins}]
                                      & Y_n \ar[r, "d^Y_n", dotted] \ar[d, dotted, "f_n"] & \cdots \ar[r, "d^Y_{2}", dotted] & Y_1 \ar[d, "f_1", dotted] \ar[r, "d^Y_1", two heads, dotted] & Y_{0} \ar[d, "\mathrlap{f_{0}}"] \\
            (X_{n+1} \ar[r, tail, "d^X_{n+1}"] & )X_n \ar[r, "d^X_n"] & \cdots \ar[r, "d^X_{2}"] & X_1 \ar[r, two heads, "d^X_1"] & X_{0}.
         \end{tikzcd}
      \end{equation}
    According to \ref{EA2op} the morphism $\sm{-d^X_1 & f_{0}}$ has an $n$-kernel and we may construct an $\sX$-deflation $d^Y_1$ and a morphism $f_1$ such that $\sm{f_1 & d^Y_1}^\top$ is a weak kernel of $\sm{-d^X_1 & f_{0}}$.
    After putting $Y_{-1} \coloneqq 0$ and $d^{Y}_{0}$ to be the zero morphism $Y_{0} \to Y_{-1}$, the axiom (E2${}^{\op}$) from \cite[Defintion 4.2]{Jas16} follows by using the claim below recursively for $m=0,\dots,n-2$.

    \begin{clm}
    Suppose for some integer $0 \leq m \leq n-2$ we are given the solid morphisms of a commutative diagram
    \[\begin{tikzcd}
    	{Y_{m+2}} & {Y_{m+1}} & {Y_{m}} & {Y_{m-1}}\\
    	{X_{m+2}} & {X_{m+1}} & {X_{m}} & 
    	\arrow["{d^Y_{m+2}}", dotted, from=1-1, to=1-2]
    	\arrow["{f_{m+2}}"', dotted, from=1-1, to=2-1]
    	\arrow["{d^Y_{m+1}}", from=1-2, to=1-3]
    	\arrow["{d^Y_{m}}", from=1-3, to=1-4]
    	\arrow["{f_{m+1}}"', from=1-2, to=2-2]
    	\arrow["{f_{m}}", from=1-3, to=2-3]
    	\arrow["{d^X_{m+2}}", from=2-1, to=2-2]
    	\arrow["{d^X_{m+1}}", from=2-2, to=2-3]
    \end{tikzcd}\]
    and the two morphisms 
    \[\text{$d^C_{m+2} \coloneqq \sm{-d^X_{m+2} & f_{m+1} \\ 0 & d^Y_{m+1}}$ \qquad and \qquad $d^C_{m+1} \coloneqq \sm{-d^X_{m+1} & f_{m} \\ 0 & d^Y_{m}}$}\]
    such that $d^C_{m+1}$ has an $(n-m)$-kernel.
    If $\sm{f_{m+1} & d^Y_{m+1}}^\top$ is a weak kernel of $d^C_{m+1}$ then 
    \begin{enumerate}[label={{(\alph*)}}]
        \item $d^C_{m+2}$ is a weak kernel of $d^C_{m+1}$, and \label{item:wk} 
        \item $d^C_{m+2}$ has an $(n-m-1)$-kernel.\label{item:m-1k}
    \end{enumerate}
    In particular, there exists a weak kernel $\sm{f_{m+2} & d^Y_{m+2}} \colon Y_{m+2} \to X_{m+2} \oplus Y_{m+1}$ of $d^C_{m+2}$, which can be chosen to be a kernel if $m=n-2$.
\end{clm}

\begin{clmproof}
    A calculation shows that $d^C_{m+1} d^C_{m+2} = 0$ using that $\sm{f_{m+1} & d^Y_{m+1}}^\top$ is a weak kernel of $d^C_{m+1}$ and $d^X_{m+1} d^X_{m+2} = 0$.
    Furthermore, as $\smash{\sm{f_{m+1} & d^Y_{m+1}}^\top}$ is a direct summand of $d^C_{m+2}$ we conclude \ref{item:wk}. 
    \Cref{item:m-1k} is a direct consequence of \Cref{lem:mktom-1k}.
\end{clmproof}
\end{proof}

Combining \Cref{lem:thesame,lem:thesame2} we conclude the main theorem of this section.
\begin{theorem}\label{thm:thesame}
   For every idempotent complete category $\cC$ the notions of $n$-exact structures on $\cC$ given in \cite[Definition 4.2]{Jas16} and \Cref{def:nexact} coincide. 
\end{theorem}

\subsection{The axioms of \texorpdfstring{$n$}{n}-exact categories in terms of functor categories}\label{sec:funapproach}
In this section we give an equivalent description of $n$-exact structures on $\cC$ in terms of certain subcategories of $\exn \cC$, that is we prove a higher analogue of \cite[Theorem B]{Eno18}.

   First we characterize all possible isomorphism closed classes $\sX \subset \KXnC$ of $n$-exact sequences which satisfy the axiom \ref{EA2op}.

   \begin{lemma}\label{lem:pb}
      Suppose $\sX \subset \KXnC$ is an isomorphism closed class of $n$-exact sequences and $\sF \coloneqq \PCn(\sX)$.
      Let $X_\bullet \in \sX$ be an $n$-exact sequence and $F \coloneqq \PCn(X_\bullet)$.
      For any $f_{0} \colon Y_{0} \to X_{0}$ the following statements are equivalent.
      \begin{enumerate}[label={{(\alph*)}}]
         \item The morphism $\sm{-d^X_1 & f_{0}}$ is an $\sX$-deflation admitting a weak kernel $\sm{f_1 & d^Y_1}^\top$ such that $d^Y_1$ is an $\sX$-deflation.\label{item:pbinA}
         \item The $\cC$-modules $\im(\aug \cov{f_{0}})$ and $F/\im(\aug \cov{f_{0}})$ are in $\sF$.\label{item:pbinmodA}
      \end{enumerate}
      Moreover, if these conditions are met then any weak kernel $\sm{f_1 & d^Y_1}^\top$ of $\sm{-d^X_1 & f_{0}}$ satisfies the additional condition that $d^Y_1$ is an $\sX$-deflation. 
   \end{lemma}
   \begin{proof}
   If there are maps such that \ref{item:pbinA} is satisfied then $\smash{F/\im( \aug \cov{f_{0}}) \cong \cok \covsm{-d^X_1 & f_{0}}}$ and $\smash{\im (\aug \cov{f_{0}}) \cong \cok \cov{d^Y_1}}$ are in $\sF$ using \Cref{lem:dia}. Hence, \ref{item:pbinmodA} holds.

   Conversely, assume \ref{item:pbinmodA} holds.
   Then $\smash{\cok \covsm{-d^X_1 & f_{0}} \cong F/\im( \aug \cov{f_{0}})}$ is in $\sF$ by \Cref{lem:dia}, showing that $\sm{-d^X_1 & f_{0}}$ is an $\sX$-deflation. 
   In particular, $\sm{-d^X_1 & f_{0}}$ has a weak kernel $\sm{f_1 & d^Y_1}^\top$.
   Then $\smash{\cok \cov{d^Y_1} \cong \im( \aug \cov{f_{0}})}$ is in $\sF$ by \Cref{lem:dia}, that is $d^Y_1$ is an $\sX$-deflation. 
   This shows that \ref{item:pbinA} and the last part of the lemma are satisfied.
   \end{proof}

    \begin{definition}\label{def:exnc}
        For a subcategory $\sF \subset \ex_n \cC$ which is closed under isomorphism let $\Pb^{\cC}_n(\sF) \subset \ex_n \cC$ be the full subcategory of all objects $F \in \sF$ such that for any $G \in \mod \cC$ and any $\alpha \colon G \to F$ both $\im \alpha$ and $F/\im \alpha$ are in $\sF$.
    \end{definition}

    Notice also that we run over all $G \in \mod \cC$ in \Cref{def:exnc} is not essential. 

    \begin{remark}\label{rem:projenough}
        As $\im \alpha = \im \alpha \beta$ and $F/\im \alpha = F/\im \alpha\beta$ for any pointwise epimorphism $\beta \colon \cov{X} \to G$ we can replace the conditions that the functors $\im \alpha$ and $F/\im \alpha$ are in $\sF$ for all $G \in \mod \cC$ in \Cref{def:exnc} equivalently by the same conditions for all $G$ of the form $\cov{X}$ for $X \in \cC$ or by the same conditions for all finitely generated $\cC$-modules $G$, that is all $\cC$-modules $G$ which admit a pointwise epimorphism $\beta \colon \cov{X} \to G$ for some $X \in \cC$.
    \end{remark}

    The operation $\Pb^\cC_n$ picks out those functors which correspond to deflations which are, in the sense of the axiom \ref{EA2op}, stable under weak pullbacks within a given class of conflations.
    In particular, using \Cref{lem:projlift} we have the following.

    \begin{corollary}\label{lem:pbcl}
        Let $\sX \subset \KXnC$ be a subclass which is closed under isomorphism and $\sF \coloneqq \PCn(\sX)$.
        Then $\sX$ satisfies the axiom \ref{EA2op} if and only if $\sF = \Pb^{\cC}_n(\sF)$.
    \end{corollary}
    \begin{proof}
        If $\sF = \Pb^{\cC}_n(\sF)$ then it follows immediately from \Cref{lem:pb} that $\sX$ satisfies \ref{EA2op}.
        If $\sX$ satisfies \ref{EA2op} then it suffices to show that $\im \alpha$ and $F/\im \alpha$ are in $\sF$ for any morphism $\alpha \colon \cov{Y_0} \to F$ with $F \in \sF$ and $Y_0 \in \cC$, by \Cref{rem:projenough}.
        If we pick $X_\bullet \in \sX$ with $F = \PCn(X_\bullet)$ then we can find a morphism $f_0 \colon Y_0 \to X_0$ with $\alpha = \cokm{\cov{d^X_1}} \cov{f_0}$, by \Cref{lem:projlift}.
        It follows then from \Cref{lem:pb} that $\im \alpha$ and $F/\im \alpha$ are in $\sF$.
    \end{proof}

    \begin{remark}
        If $\sX$ satisfies \ref{EA2op} the any weak pullback of a $\sX$-deflation is again an $\sX$-deflation by \Cref{lem:pb} and \Cref{lem:pbcl}.
        Dually, weak pushouts of $\sX$-inflations are again $\sX$-inflations if $\sX$ satisfies \ref{EA2}.
    \end{remark}

    The sequence \eqref{eq:sesextseq} in the proof of the following lemma shows that composition of deflations can, up to a small error which is hidden in \ref{EA2op}, be interpreted as taking extensions in $\exnC$.

   \begin{lemma}\label{lem:exclos}
      Suppose $\sX \subset \KXnC$ is a subclass of $n$-exact sequences which is closed under isomorphism in $\KXnC$ and satisfies \ref{EA2op}.
      Then $\sX$ satisfies \ref{EA1op} if and only if $\sF \coloneqq \PCn(\sX)$ is closed under extensions in $\exn \cC$.
   \end{lemma}
   \begin{proof}
      Suppose \ref{EA1op} is satisfied. 
      Let $0 \to F \xrightarrow{\alpha} G \xrightarrow{\beta} H \to 0$ be an pointwise exact sequence in $\mod \cC$ with $F,H \in \sF$.
      Picking a projective resolution of $H$ we obtain the solid morphisms of a diagram 
      \begin{equation}\begin{tikzcd}[ampersand replacement=\&]\label{eq:covers}
        \& {\cov{Y_{0}}} \&[.7em] {\cov{X_{0}}} \& {H} \& 0\\
      	0 \& {F} \& G \& {H} \& 0
      	\arrow[Rightarrow, no head, from=1-4, to=2-4]
      	\arrow["{\cov{f_{0}}}", from=1-2, to=1-3]
      	\arrow["\alpha", from=2-2, to=2-3]
      	\arrow["\varepsilon_G", dotted, from=1-3, to=2-3]
      	\arrow["\beta", from=2-3, to=2-4]
      	\arrow[from=2-1, to=2-2]
      	\arrow[from=2-4, to=2-5]
      	\arrow[from=1-3, to=1-4]
         \arrow[from=1-4, to=1-5]
      	\arrow["\varepsilon_F", dotted, from=1-2, to=2-2]
      \end{tikzcd}\end{equation}    
      with pointwise exact rows. The dotted arrows are constructed using \Cref{lem:projlift} such that the diagram is commutative. 
      By picking a pointwise epimorphism $\varepsilon \colon \cov{Z} \to F$ we can modify this into a commutative diagram
      \[\begin{tikzcd}[ampersand replacement=\&]
        \&[.3em] {\cov{ Y_{0} \oplus Z}} \&[3.2em] {\cov{X_{0} \oplus Z} } \& {H} \& 0\\
      	0 \& {F} \& G \& {H} \& 0
      	\arrow[Rightarrow, no head, from=1-4, to=2-4]
      	\arrow["{\covsm{f_{0} & 0 \\ 0 & \id_{Z}}}", from=1-2, to=1-3]
      	\arrow["\alpha", from=2-2, to=2-3]
      	\arrow["\sm{\varepsilon_F & \varepsilon}", from=1-2, to=2-2]
      	\arrow["\beta", from=2-3, to=2-4]
      	\arrow[from=2-1, to=2-2]
      	\arrow[from=2-4, to=2-5]
      	\arrow[from=1-3, to=1-4]
         \arrow[from=1-4, to=1-5]
      	\arrow["\sm{\varepsilon_G & \alpha \varepsilon}", from=1-3, to=2-3]
      \end{tikzcd}\]
      where $\sm{\varepsilon_G & \alpha \varepsilon}$ is a pointwise epimorphism by applying the $4$-lemma in $\Ab$ pointwise.
      Thus we may assume without loss of generality that $\varepsilon_F$ and $\varepsilon_G$ in \eqref{eq:covers} are already pointwise epimorphisms.

      As $F \in \mod \cC$ we may pick an object $Y_{1} \in \cC$ and a $\cC$-morphism $\cov{d^Y_1} \colon \cov{Y_1} \to \cov{Y_{0}}$ such that $\im \cov{d^Y_1} = \ker \smash{\varepsilon_F}$.
      It follows from pointwise evaluation and the snake lemma in $\Ab$ that the through $\smash{\cov{f_{0}}}$ induced $\cC$-morphism $\ker \varepsilon_F \to \ker \varepsilon_G$ is a pointwise epimorphism and thus that $\im \cov{f_{0} d^Y_{1}} = \ker \varepsilon_G$.
      In particular if we define $X_1 \coloneqq Y_1$ and $d^X_{1} \coloneqq f_{0} d^Y_{1}$ we have a commutative diagram
      \[\begin{tikzcd}[ampersand replacement=\&]
      	  \& {\cov{Y_{1}}} \&[.7em] {\cov{X_{1}}} \\
      	  \& {\cov{ Y_{0}}} \& {\cov{X_{0}} } \& {H} \& 0 \\
      	0 \& {F} \& G \& {H} \& 0 \\
      	  \& 0 \& 0
      	\arrow[Rightarrow, no head, from=2-4, to=3-4]
      	\arrow["{\cov{f_{0}}}", from=2-2, to=2-3]
      	\arrow["\alpha", from=3-2, to=3-3]
      	\arrow["{\varepsilon_G}", from=2-3, to=3-3]
      	\arrow["\beta", from=3-3, to=3-4]
      	\arrow[from=3-1, to=3-2]
      	\arrow[from=3-4, to=3-5]
      	\arrow[from=2-3, to=2-4]
      	\arrow[from=2-4, to=2-5]
      	\arrow["{\varepsilon_F}", from=2-2, to=3-2]
      	\arrow["{\cov{d^Y_1}}"', from=1-2, to=2-2]
      	\arrow[Rightarrow, no head, from=1-2, to=1-3]
      	\arrow["{\cov{d^X_1}}", from=1-3, to=2-3]
      	\arrow[from=3-2, to=4-2]
      	\arrow[from=3-3, to=4-3]
      \end{tikzcd}\]
      with exact rows and columns.
      As $F$ and $H$ are in $\sF$, both morphisms $d^Y_1$ and $f_{0}$ are $\sX$-deflations and so is $d^X_1 = f_{0} d^Y_1$ by \ref{EA1op}.
      This shows $G \in \sF$.

      Conversely, let $\sF$ be closed under extensions in $\mod \cC$ and suppose that we are given $\sX$-deflations $d^X_1 \colon X_1 \to X_{0}$ and $d^Y_{1} \colon Y_{1} \to Y_{0}$ with $X_{0} = Y_1$.
      By \Cref{lem:rsnake} there is a pointwise exact sequence
      \begin{equation}\label{eq:sesextseq}\scalebox{.9}{\begin{tikzcd}[ampersand replacement=\&, column sep = small]
         0 \& {{\cok \cov{d^X_1}}/\im \cokm{\hspace{.025em}\cov{d^X_1}} \kerm{\cov{d^Y_1}}} \& {\cok \cov{d^Y_1 d^X_1}} \& {\cok \cov{d^Y_1}} \& 0
	      \arrow[from=1-1, to=1-2]
	      \arrow[from=1-2, to=1-3]
	      \arrow[from=1-3, to=1-4]
	      \arrow[from=1-4, to=1-5]
      \end{tikzcd}}\end{equation}
      of $\cC$-modules. Moreover, $\ker \cov{d^Y_1}$ is in $\mod \cC$ by \Cref{lem:syzygy}, as $\cok \cov{d^Y_1} \in \mod \cC$.
      Thus the leftmost term of \eqref{eq:sesextseq} is also in $\sF$ using \Cref{lem:pbcl} and the rightmost term of \eqref{eq:sesextseq} is in $\sF$ by assumption.
      Hence, $\cok \cov{d^Y_1 d^X_1} \in \sF$ and thus $d^Y_1 d^X_1$ is an $\sX$-deflation.
   \end{proof}

   To obtain a characterisation of the axioms \ref{EA1} and \ref{EA2} we use the following natural duality $\exn \cC \xleftrightarrow{\smash{\,\sim\,}} \exn \cC^{\op}$.
   Recall the vital definitions from \Cref{lem:op,lem:resolve}.

   \begin{lemma}\label{lem:duality}
    Let $\trnc$ denote the duality $\PCnop \circ \OCn \circ \RCn \colon \exn \cC \to \exn \cC^{\op}$.
      \begin{enumerate}[label={{(\alph*)}}]
         \item We have natural equivalences $\trncop \circ \trnc \cong \Id_{\exnC}$ and $\trnc \circ \trncop \cong \Id_{\exn \cC^{\op}}$.\label{item:inv}
         \item We have $\trnc(-)(=) \cong \Ext^{n+1}_{\mod \cC}(-,\cov{=})$ naturally as functors $(\exnC)^{\op} \times \cC \to \Ab$.\label{item:natiso}
         \item The functor $\trnc$ maps pointwise exact sequences to pointwise exact sequences.\label{item:ptwexact} 
      \end{enumerate}
   \end{lemma}
   \begin{proof}
      \Cref{item:inv} follows from \Cref{lem:op,lem:resolve}.
      To show \ref{item:natiso} notice that $\cov{\RCn(F)}$ is a projective resolution of $F \in \exnC$. 
      Thus, for any $Y \in \cC$ we have 
      \begin{align*} \Ext_{\mod \cC}^{n+1}(F, \cov{Y}) &\cong \Hm^{n+1} \Hom_{\mod \cC}(\cov{\RCn(F)}, \cov{Y}) \tag{by definition of $\Ext^{n+1}_{\mod \cC}$}\\
                                                       &\cong \smash{\Hm^{n+1} \con[Y]{\RCn(F)}}  \tag{by the Yoneda lemma}\\
                                                       &= \smash{(\PCnop \circ \OCn \circ \RCn)(F)(Y)} \tag{by \Cref{rem:cohomology,rem:homology}} \\
                                                       &= \smash{\trnc(F)(Y)} \tag{by definition}
       \end{align*}
       naturally in $F \in (\exn \cC)^{\op}$ and $Y \in \cC$.
       Using that $\Ext_{\mod \cC}^{i}(F,\cov{Y}) = 0$ for all $F \in \exnC$ and $Y \in \cC$, whenever $i = n$ or $i = n+2$, we conclude \ref{item:ptwexact} as a consequence of \ref{item:natiso} and the long exact $\Ext$-sequence.
   \end{proof}

    Notice, if $\cC = \proj A$ where $A$ is a finite dimensional $\kk$-algebra then $\Tr_1^{\cC}$ can be identified with a restriction of the Auslander-Bridger transpose to a subcategory of $\mod A \cong \mod \cC$, where it is a well-defined functor.
    In fact, under the identification $\mod A \cong \mod \cC$ the category $\ex_0\, \cC$ is identified with the category of all finitely generated modules of projective dimension $1$ which have no morphisms into the regular module $A$.
 
    \begin{definition}\label{def:exncop}
        For a subcategory $\sF \subset \ex_n \cC$ which is closed under isomorphism we define the isomorphism closed subcategory $\Po_n^{\cC}(\sF) \coloneqq \trncop(\Pb_n^{\cC^{\op}}( \trnc(\sF))) $ of $\ex_n \cC$. 
    \end{definition}

    The duality $\OCn$ allows us now to characterize the axioms \ref{EA1} and \ref{EA2} similarly to \ref{EA1op} and \ref{EA2op}

    \begin{corollary}\label{cor:poextcl2}
        Let $\sX \subset \KXnC$ be a class closed under isomorphism in $\KXnC$ and define $\sF \coloneqq \PCn(\sX)$. Then
        \begin{enumerate}[label={{(\alph*)}}] 
            \item $\sX$ satisfies the axiom \ref{EA2} if and only if $\Po^{\cC}_n(\sF) = \sF$ and\label{item:parta}
            \item if $\sX$ satisfies \ref{EA2} then $\sX$ satisfies \ref{EA1} if and only if $\sF$ is closed under extensions in $\exn \cC$.
        \end{enumerate}
    \end{corollary}  
    \begin{proof}
        Essential images are invariant under auto-equivalences which are natural isomorphic to the identity functor and $\trnc$ is an exact duality.
        If we define $\sG \coloneqq \operatorname{Fun}^{\cC^{\op}}_n(\OCn(\sX))$ then $\Po^{\cC}_n(\sF) = \sF$ holds if and only if $\smash{\Pb^{\cC^{\op}}_n(\sG) = \sG}$ holds and $\sF$ is extension closed in $\exnC$ if and only if $\sG$ is extension closed in $\exnC^{\op}$.
        By \Cref{rem:opclear} the statements of the corollary are consequences of \Cref{lem:pbcl} and \Cref{lem:exclos}, respectively.
    \end{proof}
 
    The following remark in the case $n=1$ is due to Keller, cf.\s\ e.g.\s\ \cite[Appendix A]{Kel90}.
   \begin{remark}
      The axioms \ref{EA1} or \ref{EA1op} in \Cref{def:nexact} imply each other under assumption of the remaining axioms.
      In fact these axioms are then both equivalent to $\PCn(\sX)$ being closed under extensions in $\exn \cC$ by \Cref{lem:exclos} and \Cref{cor:poextcl2}.
   \end{remark}

   \begin{theorem}\label{thm:main}
      For every idempotent complete additive category $\cC$ there is an inclusion preserving one-to-one correspondence 
      \begin{align*}
         \left\{\parbox{13.2em}{\centering $n$-exact structures $(\cC,\sX)$ on $\cC$}\right\} &\xleftrightarrow{\text{\emph{1:1}}} {\left\{\parbox{18em}{\centering extension closed subcategories $\sF \subset \exn \cC$ such that $\Pb_n^{\cC}(\sF) = \sF$ and $\Po_n^{\cC}(\sF) = \sF$}\right\}}, \\
         (\cC,\sX) &\longmapsto \PCn(\sX), \\
         (\cC, \RCn(\sF)) &\longmapsfrom \sF.
      \end{align*}
   \end{theorem}
    \begin{proof}
        As $\PCn$ and $\RCn$ are quasi-inverse to each other, they induce inclusion preserving bijections between isomorphism closed subclasses of $\exn \cC$ and $\KXnC$.
        The result follows from \Cref{lem:exclos} and \Cref{lem:pbcl,cor:poextcl2}.
    \end{proof}

    \begin{remark}
        As $\sF = \exnC$ corresponds to the class $\sX = \KXnC$ of all $n$-exact sequences it follows that $\cC$ is quasi $n$-abelian if and only if $\Pb^\cC_n(\exn\cC) = \exnC$ and $\Po^\cC_n(\exn\cC) = \exnC$.
    \end{remark}

\subsection{Existence of maximal \texorpdfstring{$n$}{n}-exact structures}

   In this section we show that $\cC$ has a unique maximal $n$-exact structure.
   The following lemma is essential.
   \begin{lemma}\label{lem:pbextclosed}
      Suppose $\sF \subset \ex_n \cC$ is an isomorphism closed subclass. 
      \begin{enumerate}[label={{(\alph*)}}]
         \item If $\sF$ is extension closed in $\mod \cC$ then so is $\Pb^\cC_n(\sF)$.\label{item:PBextclosed}
      \end{enumerate}
      Moreover, let $\sG \subset \sF$ be another class which is closed under isomorphism. Then
      \begin{enumerate}[label={{(\alph*)}}, resume]
         \item the inclusions $\Pb^{\cC}_n(\sG) \subset \Pb^{\cC}_n(\sF) \subset \sF$ hold\label{item:incl}, and
         \item if $\Pb^{\cC}_n (\sF) = \sF$ and $\Po^{\cC}_n(\sG) = \sG$ hold then $\Pb^{\cC}_n(\sG) = \sG$ holds as well.\label{item:pbab}
      \end{enumerate} 
   \end{lemma}
   \begin{proof}
      We show \ref{item:PBextclosed}.
      Let $0 \to F \xrightarrow{\alpha} G \xrightarrow{\beta} H \to 0$ be a pointwise exact sequence in $\mod \cC$ with $F$ and $H$ in $\Pb^\cC_n(\sF)$.
      To show $G \in \Pb^{\cC}_n(\sF)$ we only need to show that that $\im \gamma$ and $G/\im \gamma$ are in $\sF$ for any $\cC$-morphism $\gamma \colon \cov{X} \to G$ and any $X \in \cC$, cf.\s\ \Cref{rem:projenough}.
      We can draw the solid arrows of a commutative diagram
      \[\begin{tikzcd}
      	&  &  & 0 \\
      	0 & {\ker \coimm{\beta \gamma}} & \cov{X} & {\im \beta \gamma} & 0 \\
      	0 & F & G & H & 0 \\
      	0 & {\cok \gamma'} & {\cok \gamma} & {\cok \imm{\beta\gamma}} & 0 \\
      	& 0 & 0 & 0
      	\arrow["\alpha", from=3-2, to=3-3]
      	\arrow["\beta", from=3-3, to=3-4]
      	\arrow[from=3-1, to=3-2]
      	\arrow[from=3-4, to=3-5]
      	\arrow["\gamma", from=2-3, to=3-3]
      	\arrow[from=3-3, to=4-3, "\cokm{\gamma}"]
      	\arrow[dotted, from=2-3, to=2-4, "\coimm{\beta\gamma}"]
      	\arrow[dotted, from=2-4, to=2-5]
      	\arrow[dotted, from=2-4, to=3-4, "\imm{\beta\gamma}"]
      	\arrow[dotted, from=2-2, to=2-3, "\kerm{\coimm{\beta \gamma}}"]
      	\arrow[dotted, from=2-2, to=3-2, "\gamma'"]
      	\arrow[dotted, from=4-4, to=5-4]
      	\arrow[from=4-3, to=5-3]
      	\arrow[dotted, from=4-2, to=5-2]
      	\arrow[dotted, from=3-2, to=4-2, "\cokm{\gamma'}"]
      	\arrow[dotted, from=4-2, to=4-3, "\alpha'"]
      	\arrow[dotted, from=4-3, to=4-4, "\beta'"]
      	\arrow[dotted, from=3-4, to=4-4, "\cokm{\imm{\beta\gamma}}"]
      	\arrow[dotted, from=4-4, to=4-5]
      	\arrow[dotted, from=4-1, to=4-2]
      	\arrow[dotted, from=1-4, to=2-4]
      	\arrow[dotted, from=2-1, to=2-2]
      \end{tikzcd}\]
      with pointwise exact rows and columns. 
      Image factorization of $\beta\gamma$ yields a pointwise epimorphism $\coimm{\beta\gamma}$ and a pointwise monomorphism $\imm{\beta\gamma}$.
      Moreover, $H \in \Pb^{\cC}_n(\sF)$ yields that $\im \beta\gamma$ and $\cok \beta\gamma = \cok \imm{\beta\gamma}$ are in $\sF$.
      Forming the kernel of $\coimm{\beta \gamma}$ yields a pointwise monomorphism $\kerm{\coimm{\beta \gamma}}$ and the universal property of the kernel $\alpha \colon F \to G$ yields the morphism $\gamma'$.
      Notice that $\ker{\coimm{\beta\gamma}}$ is in $\mod \cC$, using \Cref{lem:syzygy} and $\im \beta \gamma$ and $\cov{X}$ being in $\mod \cC$.
      As $F \in \Pb^\cC_{n}(\sF)$ we have that $\im \gamma'$ and $\cok \gamma'$ are in $\sF$.
      Now using the universal properties of the cokernel yields the morphisms $\alpha'$ and $\beta'$ making the diagram commutative.
      A pointwise calculation using the snake-lemma in $\Ab$ shows that the lower column in pointwise exact.
      This shows that constructed dotted arrows make the diagram commutative with pointwise exact rows and columns.

      Now notice there is a pointwise exact sequence $0 \to \ker \cokm{\gamma'} \to \ker \cokm{\gamma} \to \ker \cokm{\imm{\beta\gamma}} \to 0$ induced by the above diagram, using the universal property of the kernel and again a pointwise calculation in $\Ab$ involving the snake lemma.
      Moreover, the way we defined kernels and images we have $\ker \cokm{\gamma'} = \im \gamma'$, $\ker \cokm{\gamma} = \im \gamma$ and $\ker \imm{\beta\gamma} = \im \beta\gamma$, using that the columns are pointwise exact.
      As $\sF$ is extension closed we have that $\im \gamma$ and $\cok \gamma$ are in $\sF$ because $\cok \gamma'$, $\cok \imm{\beta\gamma}$, $\im \gamma'$ and $\im \beta \gamma$ are so.
      Thus $G \in \Pb^{\cC}_n(\sF)$.

      \Cref{item:incl} follows immediately from the definition of $\Pb^\cC_n$.
      We proceed with \ref{item:pbab}, so assume that $\Pb^\cC_n(\sF) = \sF$ and $\Po^\cC_n(\sG) = \sG$.
      Let $\gamma \colon \cov{X} \to G$ be a morphism with $G \in \sG$.
      There is a pointwise exact sequence $0 \to \im \gamma \to G \to \cok \gamma \to 0$ with terms in $\sF \subset \exnC$ as $G \in \sG \subset \sF = \Pb^{\cC}_n(\sF)$.
      Therefore we have a pointwise exact sequence 
      \[0 \to \trnc (\cok \gamma) \to \trnc (G) \to  \trnc (\im \gamma)\to 0\] 
      in $\exn \cC^{\op}$ using that $\trnc$ is an exact duality.
      Notice, $\trnc(\sG) = \Pb^{\cC^{\op}}_n(\trnc(\sG))$ is equivalent to $\sG = \Po_n^\cC(\sG)$ and therefore we have $\trnc (\cok \gamma) \in \trnc(\sG)$ and $\trnc(\im \gamma) \in \trnc(\sG)$.
      But then both $\im \gamma$ and $\cok \gamma$ have to be in $\sG$.
    \end{proof}
    
    It follows for example from \Cref{lem:pbextclosed}\ref{item:pbab} and its dual statement that in any quasi $n$-abelian category $\cC$ any isomorphism closed class of $n$-exact sequences $\sX \subset \KXnC$ satisfies \ref{EA2op} if and only if it satisfies \ref{EA2}.

    To show that there is a unique maximal $n$-exact structure on $\cC$ we apply $\Pb^{\cC}_n$ and $\Po^{\cC}_n$ recursively to $\exnC$ to obtain a category that is actually stable under the application of $\Pb^{\cC}_n$ and $\Pb^\cC_n$.
    
    \begin{definition}
        For an isomorphism closed subcategory $\sF \subset \exnC$ and any $i \in \NN$ let 
        \begin{enumerate}[label={{(\alph*)}}]
            \item $(\Pb^\cC_n)^i(\sF)$ be the $i$-fold application of $\Pb^{\cC}_n$ to $\sF$, and 
            \item $(\Pb^{\cC}_n)^\infty(\sF)$ be the intersection $\bigcap_{i \in \NN} (\Pb^{\cC}_n)^i(\sF)$.
        \end{enumerate}
        We define $(\Po^\cC_n)^i(\sF)$ and $(\Pb^{\cC}_n)^\infty(\sF)$ similarly and put $\mathsf{max}_n \cC \coloneqq (\Po^{\cC}_n)^\infty((\Pb^{\cC}_n)^\infty(\exnC) )$.
    \end{definition}

    We show that $\mathsf{max}_n \cC$ yields the unique maximal $n$-exact structure on $\cC$ via \Cref{thm:main}.

   \begin{proposition}\label{prop:actualmain}
      The subcategory $\mathsf{max}_n \cC \subset \exnC$ is extension closed and satisfies $\Pb^{\cC}_n(\mathsf{max}_n \cC) = \mathsf{max}_n \cC$ and  $\Po^{\cC}_n(\mathsf{max}_n \cC) = \mathsf{max}_n \cC$. 
      If an isomorphism closed subcategory $\sF \subset \exnC$ satisfies $\Pb^\cC_n(\sF) = \sF$ and $\Po^{\cC}_n(\sF) = \sF$ then $\sF \subset \mathsf{max}_n \cC$.
   \end{proposition}
    \begin{proof}
        It follows from inductively applying \Cref{lem:pbextclosed}\ref{item:PBextclosed} that $(\Pb^{\cC}_n)^{i}(\exnC)$ is extension closed in $\exn \cC$ for $i \in \NN$.
        The same is true for $(\Pb^{\cC}_n)^{\infty}(\exnC)$ as the intersection of extension closed subcategories is again extension closed.
        If $F \in (\Pb^{\cC}_n)^{\infty}(\exnC)$ then $F \in (\Pb^{\cC}_n)^{i}(\exnC)$ for all $i \in \NN$.
        If $\alpha \colon G \to F$ is a $\cC$-morphism with $G \in \mod \cC$ then both $\im \alpha$ and $\cok \alpha$ are in $(\Pb^{\cC}_n)^{i-1}(\exnC)$ for all $i \in \NN_{\geq 1}$ and therefore also in $(\Pb^{\cC}_n)^{\infty}(\exnC)$ as the latter is the intersection over a descending chain, see \Cref{lem:pbextclosed}\ref{item:incl}.
        That means that 
        \begin{equation}\label{eq:PBCLOSED}\Pb^{\cC}_n((\Pb^{\cC}_n)^{\infty}(\exnC)) = (\Pb^{\cC}_n)^{\infty}(\exnC)\end{equation}
        holds.
        As $\trnc$ is an exact duality the same arguments applied to $(\Pb^{\cC}_n)^{\infty}(\exnC)$ instead of $\exnC$ show that $\mathsf{max}_n \cC$ is extension closed in $\exnC$ and that $\Po^{\cC}_n(\mathsf{max}_n \cC) = (\mathsf{max}_n \cC)$
        holds.
        It follows from \eqref{eq:PBCLOSED} and \Cref{lem:pbextclosed}\ref{item:pbab} that $\Pb^{\cC}_n(\mathsf{max}_n \cC) = \mathsf{max}_n \cC$. 
        
        If $\sF \subset \exnC$ satisfies $\Pb^{\cC}_n(\sF) = \sF$ then \Cref{lem:pbextclosed}\ref{item:incl} shows that $\sF \subset (\Pb^\cC_n)^i({\exnC})$ for all $i \in \NN$ inductively.
        Thus also $\sF \subset (\Pb^\cC_n)^\infty({\exnC})$ holds in this case.
        The same argument involving $\Po^{\cC}_n$ shows that $\sF \subseteq \mathsf{max}_n \cC$ if additionally $\Po^{\cC}_n(\sF) = \sF$.
   \end{proof}

    \begin{theorem}\label{thm:exmax}
       Every idempotent complete additive category $\cC$ has a unique maximal $n$-exact structure.
       This structure corresponds to $\mathsf{max}_n \cC$ via  \Cref{thm:main}, that is its conflations are given by all $(n+1)$-term complexes $X_\bullet$ whose Yoneda embedding $\cov{X_\bullet}$ is a projective resolution of a functor in $\mathsf{max}_n \cC$. 
    \end{theorem}
    \begin{proof}
        There is an $n$-exact structure which corresponds to $\mathsf{max}_n \cC$ by the first part of \Cref{prop:actualmain}.
        As the class of $\sX$-conflations of an $n$-exact category $(\cC,\sX)$ satisfies \ref{EA2} and \ref{EA2op} the maximality of this structure follows from \Cref{cor:better} below.
    \end{proof}

    The following application of \Cref{thm:exmax} is useful when constructing $n$-exact structures containing a certain class of $n$-exact sequences as the axioms \ref{EA1} and \ref{EA1op} seem to be usually harder to verify than \ref{EA2} and \ref{EA2op}.
    \begin{corollary}\label{cor:better}
        Let $\sX \subset \KXnC$ be an isomorphism closed class of $n$-exact sequences satisfying the axioms \ref{EA2} and \ref{EA2op}.
        Then all $\sX$-conflations are also conflations in the maximal $n$-exact structure on $\cC$.
    \end{corollary}
    \begin{proof}
        By \Cref{lem:pbcl,cor:poextcl2}\ref{item:parta} the subcategory $\sF \coloneqq \PCn(\sX)$ of $\exnC$ satisfies the hypotheses of \Cref{prop:actualmain} and thus $\sF \subset \mathsf{max}_n \cC$.
        The result follows as the one-to-one correspondence in \Cref{thm:main} is inclusion preserving.
    \end{proof}

\begin{acknowledgment}
    I would like to greatly thank my PhD supervisor Peter Jørgensen and my postdoctoral supervisor Steffen König for many helpful comments, discussions and suggestions concerning this paper.
    Moreover, I would like to thank Vitor Gulisz for helpful comments and information regarding his paper.
    Many thanks also to Sondre Kvamme for helpful discussions regarding this paper, in particular for discussions about functor categories and \Cref{def:nexact}. 
\end{acknowledgment}

      \bibliographystyle{alpha}

\begin{thebibliography}{GKO13}
         \bibitem[Büh10]{Bueh10}
         \newblock Theo Bühler,
         \emph{Exact categories},
         Expositiones Mathematicae,
         \textbf{28}(1):1-69,
         2010.
        
        \bibitem[Cri12]{Cri12}
        \newblock Septimiu Crivei,
        \emph{Maximal exact structures on additive categories revisited},
        Mathematische Nachrichten,
        \textbf{285}(4):440-446,
        2012.

        \bibitem[Eno17]{Eno17}
        \newblock Haruhisa Enomoto,
        \emph{Classifying exact categories via Wakamatsu tilting},
        Journal of Algebra,
        \textbf{485}:1-44,
        2017.

        \bibitem[Eno18]{Eno18}
        \newblock Haruhisa Enomoto,
        \emph{Classifications of exact structures and Cohen–Macaulay-finite algebras},
        Advances in Mathematics,
        \textbf{335}:838-877,
        2018.

        \bibitem[GKO13]{GKO13}
        \newblock Christof Geiss, Bernhard Keller and Steffen Oppermann,
        \emph{n-angulated categories},
        Journal für die reine und angewandte Mathematik,
        \textbf{2013}(675):101-120,
        2013.

        \bibitem[Gul24]{Gul24}
        \newblock Vitor Gulisz,
        \emph{A functorial approach to $n$-abelian categories},
        Preprint, 
        arXiv:2409.10438,
        2024.

         \bibitem[HLN21]{HLN21}
         \newblock Martin Herschend, Yu Liu and Hiroyuki Nakaoka, 
         \emph{$n$-exangulated categories (I): Definitions and fundamental properties},
         Journal of Algebra,
         \textbf{570}:531-586,
         2021.

         \bibitem[Jas16]{Jas16}
         \newblock Gustavo Jasso, 
         \emph{$n$-Abelian and $n$-exact categories},
         Mathematische Zeitschrift,
         \textbf{283}:703–759(3-4),
         2016.

        \bibitem[Kel90]{Kel90}
        \newblock Bernhard Keller,
        \emph{Chain complexes and stable categories},
        manuscripta mathematica,
        \textbf{67}(1):379-417,
        1990.

        \bibitem[Kva24]{Kva24}
        \newblock Sondre Kvamme, 
        Private communication,
        2024.

        \bibitem[Rum11]{Rum11}
        \newblock Wolfgang Rump,
        \emph{On the maximal exact structure of an additive category}
        Fundamenta Mathematicae,
        \textbf{214}(1):77-87,
        2011.

        \bibitem[Sie11]{Sie11}
        \newblock Dennis Sieg,
        \emph{Maximal exact structures on additive categories},
        Mathematische Nachrichten,
        \textbf{284}(16):2093-2100,
        2011.
      \end{thebibliography}
      
\end{document}